\documentclass[11pt]{article}
\usepackage{amsfonts,amssymb,amsbsy,amsmath,amsthm,mathrsfs,enumerate,verbatim}
 \usepackage[colorlinks=true]{hyperref}
 \hypersetup{urlcolor=blue, citecolor=red}
\newtheorem{theorem}{Theorem}[section]
\newtheorem{proposition}[theorem]{Proposition}
\newtheorem{lemma}[theorem]{Lemma}
\newtheorem{follow}[theorem]{Corollary}

\theoremstyle{definition}
\newtheorem{remark}[theorem]{Remark}

\newcommand{\bel}{\begin{equation} \label}
\newcommand{\ee}{\end{equation}}

\newcommand{\pd}{\partial}

\newcommand{\Dom}{{\textit{D}}}

\newcommand{\R}{{\mathbb R}}
\newcommand{\N}{{\mathbb N}}

\newcommand{\bx}{{\bf x}}

\newcommand{\txi}{\tilde{\xi}}

\newcommand{\eps}{{\epsilon}}
\def\beq{\begin{equation}}
\def\eeq{\end{equation}}
\newcommand{\bea}{\begin{eqnarray}}
\newcommand{\eea}{\end{eqnarray}}
\newcommand{\beas}{\begin{eqnarray*}}
\newcommand{\eeas}{\end{eqnarray*}}

{

\newcommand{\abs}[1]{\left\lvert#1\right\rvert}
\newcommand{\norm}[1]{\left\lVert#1\right\rVert}

 \usepackage{graphicx,color}
 \definecolor{mygreen}{cmyk}{1,0,1,0.1}

\usepackage{epsf}
\usepackage{epstopdf}

\usepackage{pdfsync}

\begin{document}

\begin{center}
{\Large \bf Determining the waveguide conductivity in a hyperbolic equation from a single measurement on the lateral boundary}

\medskip

\end{center}

\medskip

\begin{center}
{\sc \footnote{Aix-Marseille Universit\'e, CNRS, I2M UMR 7373, \'Ecole Centrale. E-mail: michel.cristofol@univ-amu.fr.}{Michel Cristofol}, \footnote{University of Science and Technology of China, School of Mathematical Sciences, P. O. Box 4, Hefei, Anhui 230026, People's Republic of China. E-mail: shuminli@ustc.edu.cn.}{Shumin Li}, \footnote{Aix-Marseille Universit\'e, CNRS, CPT UMR 7332, 13288 Marseille, France \& Universit\'e de Toulon, CNRS, CPT UMR 7332, 83957 La Garde, France. E-mail: eric.soccorsi@univ-amu.fr.}{Eric Soccorsi}}
\end{center}

\begin{abstract}
We consider the multidimensional inverse problem of determining the conductivity coefficient of a hyperbolic equation in an infinite cylindrical domain, from a single boundary observation of the solution. We prove H\"older stability with the aid of a Carleman estimate specifically designed for hyperbolic waveguides.
\end{abstract}

\medskip




\section{Statement of the problem and results}
\label{intro} 
\setcounter{equation}{0}
\subsection{Introduction}
The present paper deals with the inverse problem of determining the time-independent isotropic conductivity coefficient $c : \Omega \to \R$ appearing in the hyperbolic partial differential equation $\partial_t^2 -\nabla \cdot c \nabla=0$, where $\Omega := \omega \times \R$ is an infinite cylindrical domain whose cross section $\omega$ is a bounded open subset of $\R^{n-1}$, $n \geq 2$. Namely, $\ell>0$ being arbitrarily fixed, we seek H\"older stability in the identification of $c$ in $\Omega_{\ell}:= \omega \times (-\ell,\ell)$ from the observation of $u$ on the lateral boundary $\Gamma_L := \partial \omega \times (-L,L)$ over the course of time $(0,T)$, for $L > \ell$ and $T>0$ sufficiently large.

Several stability results in the inverse problem of determining one or several unknown coefficients of a hyperbolic equation from a finite number of measurements of the solution are available in the mathematics literature \cite{Be5,BJY08,BY06,BeYa08,IY2,IY03, KY06, Y99}. Their derivation relies on Bukhgeim-Klibanov's method \cite{BK}, which is by means of a Carleman inequality specifically designed for hyperbolic systems. More precisely,
\cite{IY2, Y99} study the determination of the zero-th order term $p : \Omega \to \R$ appearing in $\partial_t -\Delta + p=0$, while \cite{Be5,BeYa08} deal with the identification of the speed $c : \Omega \to \R$ in the hyperbolic equation $\partial_t -c A =0$ where $A=A(x,D_x)$ is a second order differential operator. The case of a principal matrix term in the divergence form, arising from anisotropic
media, was treated by Bellassoued, Jellali and Yamamoto in \cite{BJY08}, using the full data (i.e. the measurements are performed on the whole boundary). Using the FBI transform
Bellassoued and Yamamoto claimed logarithmic stability in \cite{BY06} from arbitrarily small boundary observations. Imanuvilov and Yamamoto derived stability results in \cite{IY03} by means of $H^{-1}$ Carleman inequality, from data observation on subdomains fulfilling specific geometric assumptions. In \cite{KY06} Klibanov and Yamamoto employed a different approach inspired by \cite{KM91}. 

Similarly, numerous authors have used the Dirichlet-to-Neumann operator to claim stability in the determination of unknown coefficients of a hyperbolic equation.
We refer to \cite{BCY09, IS, SU} for a non-exhaustive list of such references.

In all the above mentioned papers, the domain was bounded. Recently, the Bukhgeim-Klibanov method was adapted to the framework of infinite quantum cylindrical domains in \cite{CCG, KPS1, KPS2}. Inverse boundary value problems stated in unbounded waveguides were also studied in \cite{CKS1, CKS2, CS} with the help of the Dirichlet-to-Neumann operator.
In all the six previous articles, the observation is taken on the infinitely extended lateral boundary of the waveguide. 
The approach developed in this paper is rather different in the sense that we aim to retrieve any arbitrary bounded subpart of the non-compactly supported conductivity $c$ from  one data taken on a compact subset of the lateral boundary. This requires that suitable smoothness properties of the solution to \eqref{S1} be preliminarily established in the context of the unbounded domain $\Omega$.

The paper is organized as follows. Section \ref{sec-direct} is devoted to the analysis of the direct problem associated with the hyperbolic system under study. In Section \ref{sec-Carleman} we prove a global Carleman estimate specifically designed for hyperbolic systems in the cylindrical domain $\Omega$. Finally Section \ref{sec-inverse} contains the analysis of the inverse problem and the proof of the main result.

\subsection{Settings}
\label{sec-intro} 
\setcounter{equation}{0}

\subsubsection{Notations}
Throughout this text we write $x=(x',x_n) \in \Omega$ for every $x':=(x_1,\ldots,x_{n-1}) \in \omega$ and $x_n \in \R$. 
Further, we denote by $| y |:=\left( \sum_{i=1}^m y_j^2 \right)^{1 \slash 2}$ the Euclidian norm of $y=(y_1,\ldots,y_m) \in \R^m$, $m \in \N^*$, and we put $\mathbb{S}^{n-1}:=\left\{ x'=(x_1,\ldots,x_{n-1}) \in \R^{n-1},\ |x'|=1 \right\}$.

For the sake of shortness we write $\pd_j$ for $\pd \slash \pd x_j$, $j \in \N_{n+1}^*:=\{ m \in \N^*,\ m \leq n+1 \}$. 
For convenience the time variable $t$ is sometimes denoted by $x_{n+1}$ so that $\pd_t=\pd \slash \pd t=\pd_{n+1}$. 
We set $\nabla:=(\pd_1,\ldots,\pd_n)^T$, $\nabla_{x'}:=(\pd_1,\ldots,\pd_{n-1})^T$ and
$\nabla_{x,t}=(\pd_1,\ldots,\pd_{n},\pd_t)^T$. 

For any open subset $D$ of $\R^m$, $m \in \N^*$, we denote by $H^p(D)$ the $p$-th order Sobolev space on $D$ for every $p \in \N$, where
$H^0(D)$ stands for $L^2(D)$. We write $\langle \cdot , \cdot \rangle_{p,D}$ (resp., $\| \cdot \|_{p,D}$) for the usual scalar product (resp., norm) in $H^p(D)$ and we denote by
$H_0^1(D)$ the closure of $C_0^{\infty}(D)$ in the topology of $H^1(D)$

Finally, for $d>0$ we put $\Omega_d := \omega \times (-d,d)$, $Q_d:=\Omega_d \times(0, T)$, $\Gamma_{d}:= \partial \omega \times (-d, d)$ and $\Sigma_{d} := \partial \omega \times(-d, d) \times (0, T)$.

\subsubsection{Statement of the problem}

We examine the following initial boundary value problem (IBVP in short)
\bel{S1}
\left\{  \begin{array}{ll} \pd_t^2 u -  \nabla \cdot c \nabla u = 0 & \mbox{in}\ Q := \Omega \times (0,T), \\ 
u(\cdot,0) = \theta_0,\ \pd_t u(\cdot,0)  =  \theta_1 & \mbox{in}\ \Omega, \\
u = 0 &  \mbox{on}\ \Sigma:= \Gamma \times (0,T),
\end{array} \right.
\ee
with initial conditions $(\theta_0,\theta_1)$, where $c$ is the unknown conductivity coefficient we aim to retrieve.
This is by means of the Bukhgeim-Klibanov method imposing that the solution $u$ to \eqref{S1} be sufficiently smooth and appropriately bounded. 

Throughout the entire text we shall suppose that $c$ fulfills the ellipticity condition
\bel{c0a}
c \geq c_m\ \mbox{in}\ \Omega,
\ee
for some positive constant $c_m$. Notice that we may assume, and this will be systematically the case in the sequel, without restricting the generality of the foregoing, 
that $c_m \in (0,1)$. 

Let us now say a few words on the solution to \eqref{S1}.
In order to exhibit sufficient conditions on the initial conditions $(\theta_0,\theta_1)$ (together with the cross section $\omega$ and the conductivity $c$) ensuring that the solution to \eqref{S1} is within an appropriate functional class we shall make precise further, we need to introduce the self-adjoint operator $A=A_c$, associated with $c$, generated in $L^2(\Omega)$ by the closed sesquilinear form
$$ q_A[u]:= \| c^{1 \slash 2} u \|_{0,\Omega}^2=\int_{\Omega}  c(x) | u(x) |^2 dx,\ u \in \Dom(q_A):=H_0^1(\Omega). $$
Evidently, $A$ acts on its domain as $-\nabla \cdot c \nabla$. Since $A$ is positive in $L^2(\Omega)$, by \eqref{c0a}, the operator $A^{1 \slash 2}$ is well defined from the spectral theorem, and $\Dom(A^{1 \slash 2})=\Dom(q_A)=H_0^1(\Omega)$. For the sake of definiteness, we set $A^0:=I$ and $\Dom(A^0):=L^2(\Omega)$, where $I$ denotes the identity operator in $L^2(\Omega)$, and we put
$$ A^{m \slash 2} v := A^{(m-1) \slash 2} (  A^{1 \slash 2} v ),\ v \in \Dom(A^{m \slash 2}) := \{ v \in \Dom(A^{(m-1) \slash 2}),\ A^{1\slash 2} v \in \Dom(A^{(m-1) \slash 2}) \}, $$
for each $m \in \N^*$. It turns out that the linear space $\Dom(A^{m \slash 2})$ endowed with the scalar product
$$ \langle v ,w \rangle_{\Dom(A^{m \slash 2})} := \sum_{j=0}^m \langle A^{j \slash 2} v , A^{j \slash 2} w \rangle_{0,\Omega}, $$
is Hilbertian, and it is established in Proposition \ref{pr-dom} that
\bel{z00}
\Dom(A^m) = \{ v \in H^{2m}(\Omega);\ v, Av, \ldots, A^{p-1} v \in H_0^1(\Omega) \},\ m \in \{p-1 \slash 2, p \},\ p \in \N^*,
\ee
provided $\pd \omega$ is $C^{2m}$ and $c \in W^{2m-1,\infty}(\Omega)$.
As a matter of fact we know from Corollary \ref{cor-eur} for any  natural number $m$, that the system \eqref{S1} admits a unique solution 
\bel{z0}
u \in \bigcap_{k=0}^{m+1} C^k([0,T]; H^{m+1-k}(\Omega)),
\ee 
provided the boundary $\pd \omega$ is $C^{m+1}$, the conductivity $c \in W^{m,\infty}(\Omega;\R)$ fulfills \eqref{c0a} and $(\theta_0,\theta_1) \in \Dom(A^{(m+1) \slash 2}) \times \Dom(A^{m \slash 2})$. 
Moreover, if $\| c \|_{W^{m,\infty}(\Omega)} \leq c_M$
for some constant $c_M>0$, then the solution $u$ to \eqref{S1} satisfies the estimate
\bel{z2}
\sum_{k=0}^{m+1} \| u \|_{C^k([0,T];H^{m+1-k}(\Omega))} \leq C \left( \| \theta_0 \|_{m+1,\Omega} + \| \theta_1 \|_{m,\Omega} \right),
\ee
where $C>0$ depends only on $T$, $\omega$ and $c_M$.

\subsubsection{Admissible conductivity coefficients and initial data}
In order to solve the inverse problem associated with \eqref{S1} we seek solutions belonging to $\cap_{k=3}^4 C^k([0,T];H^{5-k}(\Omega))$. Hence we chose $m=4$ in \eqref{z0} and impose on $c$ to be in $W^{4,\infty}(\Omega;\R)$ and satisfy
\eqref{c0a}. In what follows we note $c_M$ a positive constant fulfilling
\bel{c0b}
 \| c \|_{W^{4,\infty}(\Omega)} \leq c_M.
\ee
Since our strategy is based on a Carleman estimate for the hyperbolic system \eqref{S1}, it is also required that the condition
\bel{c1}
a' \cdot \nabla_{x'} c \geq \mathfrak{a}_0\ \mbox{in}\ \Omega,
\ee
hold for some $a'=(a_1,\ldots,a_{n-1}) \in \mathbb{S}^{n-1}$ and $\mathfrak{a}_0>0$.
Hence, given $\mathcal{O}_*$, a neighborhood of $\Gamma$ in $\R^{n-1}$, and $c_* \in W^{4,\infty}(\mathcal{O}_* \cap \Omega;\R)$ satisfying
\bel{c0a*}
c_* \geq c_m\ \mbox{and}\ a' \cdot \nabla_{x'} c_* \geq \mathfrak{a}_0\ \mbox{in}\ \mathcal{O}_* \cap \Omega,
\ee
we introduce the set $\Lambda_{\mathcal{O}_*}=\Lambda_{\mathcal{O}_*}(a',\mathfrak{a_0},c_*,c_m,c_M)$ of admissible conductivity coefficients as
\bel{1.7}
\Lambda_{\mathcal{O}_*}:=\{ c \in W^{4,\infty}(\Omega;\R)\ \mbox{obeying}\ \eqref{c0a}\ \mbox{and}\ \eqref{c0b}-\eqref{c1};\ c=c_*\ \mbox{in}\ \mathcal{O}_* \cap \Omega \}.
\ee

Notice that the above choice of $m=4$ dictates that $(\theta_0,\theta_1)$ be taken in $\Dom(A^{5 \slash 2}) \times \Dom(A^2)$, which is embedded in $H^5(\Omega) \times H^4(\Omega)$ according to \eqref{z00}.
Furthermore, it is required by the analysis of the inverse problem carried out in this article that $\theta_0$ be in $W^{3,\infty}(\Omega)$ and satisfy
\bel{S10}
- a' \cdot \nabla_{x'} \theta_0 \geq \eta_0\ \mbox{in}\ \omega_*,
\ee
for some $\eta_0>0$ and some open subset $\omega_*$ in $\R^{n-1}$, with $C^2$ boundary, satisfying
\bel{ic0} 
\overline{\omega \setminus \mathcal{O}_*} \subset \omega_* \subset \omega.
\ee
Thus, for $M_0>0$ such that
\bel{S100}
\| \theta_0 \|_{W^{3,\infty}(\Omega)} + \sum_{j=0}^1 \| \theta_j \|_{5-j,\Omega} \leq M_0,
\ee
we define the set $\Theta_{\omega_*}=\Theta_{\omega_*}(a',M_0,\eta_0)$
of admissible initial conditions $(\theta_0,\theta_1)$ as
\bel{ic2}
\Theta_{\omega_*}:= \left\{ (\theta_0, \theta_1)\in \left( \Dom(A^{5 \slash 2}) \cap W^{3,\infty}(\Omega) \right) 
\times \Dom(A^2),\ \mbox{fulfilling}\ \eqref{S10}\ \mbox{and}\ \eqref{S100} \right\}.
\ee

Having introduced all these notations we may now state the main result of this paper.

\subsubsection{Main result}
The following result claims H\"older stability in the inverse problem of determining $c$ in $\Omega_\ell$, where $\ell>0$ is arbitrary, from the knowledge of one boundary measurement  of the solution to \eqref{S1}, performed on $\Sigma_L$ for $L > \ell$ sufficiently large. The corresponding observation is viewed as a vector of the Hilbert space
$$\mathscr{H}( \Sigma_{L}) :=H^4(0,T;L^2(\Gamma_{L}))\cap H^3(0,T;H^1(\Gamma_{L})), $$
endowed with the norm,
$$ \norm{u}^2_{\mathscr{H}( \Sigma_{L})} :=\norm{u}^2_{H^4(0,T;L^2(\Gamma_{L}))}+\norm{u}^2_{H^3(0,T;H^1(\Gamma_{L}))},\ u\in
\mathscr{H}( \Sigma_{L}). $$

\begin{theorem}
\label{T.1} 
Assume that $\pd \omega$ is $C^5$ and let $\mathcal{O}_*$ be a neighborhood of $\Gamma$ in $\R^{n-1}$. For $a'=(a_1,\ldots,a_{n-1}) \in \mathbb{S}^{n-1}$, $\mathfrak{a}_0>0$, $c_m \in (0,1)$, $c_M>c_m$ and $c_* \in W^{4,\infty}(\mathcal{O}_* \cap \Omega;\R)$ fulfilling \eqref{c0a*}, pick $c_j$, $j=1,2$, in $\Lambda_{\mathcal{O}_*}(a',\mathfrak{a_0},c_*,c_m,c_M)$, defined by \eqref{1.7}. Further, given $M_0>0$, $\eta_0>0$ and an open subset $\omega_* \subset \R^{n-1}$ obeying \eqref{ic0}, chose
$\theta_0 \in \Theta_{\omega_*}(a',M_0,\eta_0)$, defined in \eqref{ic2}, and $\theta_1 = 0$. 

Then for any $\ell>0$ we may find $L >\ell$ and $T>0$,  
such that the $\bigcap_{k=0}^5 C^k([0,T],H^{5-k}(\Omega))$-solution $u_j$, $j=1,2$, to \eqref{S1} associated with $(\theta_0,\theta_1)$, where $c_j$ is substituted for $c$, satisfies 
$$
\norm{c_1-c_2}_{H^1(\Omega_\ell)} \leq C  \norm{ u_1-u_2 }_{\mathscr{H}( \Sigma_{L})}^\kappa.
$$
Here $C>0$ and $\kappa \in (0,1)$ are two constants depending only on $\omega$, $\ell$, $M_0$, $\eta_0$, $a'$, $\mathfrak{a}_0$, $c_\star$, $c_m$ and $c_M$.
\end{theorem}

We stress out that the measurement of the observation data is performed on $\Gamma_{L}$ and not on the whole boundary $\partial \Omega_{L}$.





\section{Analysis of the direct problem}
\label{sec-direct}
\setcounter{equation}{0}

In this section we establish existence and uniqueness results as well as regularity properties, for the solution to hyperbolic \eqref{S1}-like IBVP systems.
The corresponding results are similar to the ones obtained for hyperbolic equations in bounded domains (see e. g. \cite[Sect. 7.2, Theorem 6]{Ev}) but since $\Omega$ is unbounded here, they cannot be derived from them. 

\subsection{Existence and uniqueness result}
With reference to \eqref{S1} we consider the boundary value problem
\bel{r0}
\left\{  \begin{array}{ll} \pd_t^2 v + A v = f & \mbox{in}\ Q \\ 
v(\cdot,0) = g,\ \pd_t v(\cdot,0)  =  h & \mbox{in}\ \Omega,
\end{array} \right.
\ee
where $f$, $g$ and $h$ are suitable data, and we recall from \footnote{Upon taking $H:=D(A^0)=L^2(\Omega)$, $V:=D(A^{1 \slash 2})=H_0^1(\Omega)$ and $a(t;u,v):=\langle A^{1 \slash 2} u , A^{1 \slash 2} v \rangle_{0,\Omega} = \int_{\Omega} c(x) \nabla u(x) \cdot \overline{\nabla v(x)} dx$ for $u,v \in V$ and all $t \in [0,T]$.}{\cite[Sect. 3, Theorem 8.2]{LM1}} the following existence and uniqueness result.

\begin{proposition}
\label{pr-r1}
Assume that $f \in C^0([0,T];\Dom(A^0))$, $g \in \Dom(A^{1 \slash 2})$ and $h \in \Dom(A^0)$. Then there exists a unique solution $v$ to \eqref{r0} such that
\bel{r1a}
\pd_t^k  v \in C^0([0,T];\Dom(A^{(1-k) \slash 2})),\ k=0,1.
\ee
Moreover we have the estimate
\bel{r1b}
\sum_{k=0}^1 \sup_{t \in [0,T]} \| \pd_t^k v(\cdot,t) \|_{\Dom(A^{(1-k) \slash 2})} \leq C \left( \| f \|_{C^0([0,T];\Dom(A^0))} + \| g \|_{\Dom(A^{1 \slash 2})} + \| h \|_{\Dom(A^0)} \right).
\ee
\end{proposition}

\subsection{Improved regularity}

\begin{proposition}
\label{pr-r2}
Assume that $f \in C^1([0,T];\Dom(A^0))$, $g \in \Dom(A)$ and $h \in \Dom(A^{1 \slash 2})$. Then the solution $v$ to \eqref{r0} satisfies
\bel{r2a}
\pd_t^k  v \in C^0([0,T];\Dom(A^{(2-k) \slash 2})),\ k=0,1,2,
\ee
and we have the estimate
\bel{r2b}
\sum_{k=0}^2 \sup_{t \in [0,T]} \| \pd_t^k v(\cdot,t) \|_{\Dom(A^{(2-k) \slash 2})} \leq C \left( \| f \|_{C^1([0,T];\Dom(A^0))} + \| g \|_{\Dom(A)} + \| h \|_{\Dom(A^{1 \slash 2})} \right).
\ee
\end{proposition}
\begin{proof}
By differentiating \eqref{r0} with respect to $t$, we check that $w:=\pd_t v$ obeys
\bel{r2c}
\left\{  \begin{array}{ll} \pd_t^2 w + A w = \pd_t f & \mbox{in}\ Q \\ 
w(\cdot,0) = h,\ \pd_t w(\cdot,0)  =  f(\cdot,0) - A g & \mbox{in}\ \Omega.
\end{array} \right.
\ee
Since $f(\cdot,0)-Ag$ is lying in $\Dom(A^0)$ then we have $\pd_t^{k+1} v = \pd_t^k w \in C^0([0,T];\Dom(A^{(1-k) \slash 2}))$ for $k=0,1$, with the estimate
\bea
\sum_{k=0}^1 \sup_{t \in [0,T]} \| \pd_t^{k+1} v(\cdot,t) \|_{\Dom(A^{(1-k) \slash 2})} & \leq & C \left( \| \pd_t f \|_{C^0([0,T];\Dom(A^0))} + \| h \|_{\Dom(A^{1 \slash 2})} +
\| f(\cdot,0) - A g \|_{\Dom(A^0)} \right)  \nonumber \\
& \leq &  C \left( \| f \|_{C^1([0,T];\Dom(A^0))} + 
\| g \|_{\Dom(A)} + \| h \|_{\Dom(A^{1 \slash 2})} \right) \label{r2d}
\eea
by Proposition \ref{pr-r1}. Further, as $Av = f - \pd_t^2 v$ from the first line in \eqref{r0}, we get that $v \in C^0([0,T];\Dom(A))$, and that
$\| v(\cdot,t) \|_{\Dom(A)}$ is majorized by the right hand side of \eqref{r2d}, uniformly in $t \in [0,T]$. This and \eqref{r2d} yield the desired result.
\end{proof}

\subsection{Higher regularity}

\begin{theorem}
\label{thm-hr}
Let $m$ be a nonnegative integer. We assume that $g \in \Dom(A^{(m+1) \slash 2})$, $h \in \Dom(A^{m \slash 2})$, and
$$ \pd_t^k f \in C^0([0,T];\Dom(A^{(m-k) \slash 2})),\ k=0,\ldots,m. $$
Then there exists a unique solution $v$ to \eqref{r0}, such that
\bel{hr0}
\pd_t^k v \in C^0([0,T];\Dom(A^{(m+1-k) \slash 2})),\ k=0,1,\ldots,m+1.
\ee
Moreover we have the estimate
\bea
& & \sum_{k=0}^{m+1} \sup_{t \in [0,T]} \| \pd_t^k v(\cdot,t) \|_{\Dom(A^{(m+1-k) \slash 2})} \nonumber \\
& \leq & C \left( \sum_{k=0}^m \| \pd_t^k f \|_{C^0([0,T];\Dom(A^{(m-k) \slash 2}))} + \| g \|_{\Dom(A^{(m+1) \slash 2})} + \| h \|_{\Dom(A^{m \slash 2})} \right).
\label{hr1}
\eea
\end{theorem}
\begin{proof}
\noindent a) The proof is by an induction on $m$, the case $m=0$ following from Proposition \ref{pr-r1}. 

\noindent b) We assume that the theorem is valid for some $m \in \N$ and suppose that
\bel{hr2}
\left\{ \begin{array}{l}
g \in \Dom(A^{(m+2) \slash 2}),\ h \in \Dom(A^{(m+1) \slash 2}),\\ \pd_t^k f \in C^0([0,T];\Dom(A^{(m+1-k) \slash 2})),\ k=0,\ldots,m+1. \end{array} \right.
\ee
We use the same strategy as in the proof of Proposotion \ref{pr-r2}. That is we differentiate \eqref{r0} with respect to $t$ and get that $w:=\pd_t v$ is solution to \eqref{r2c}. 
Next, using that $h \in \Dom(A^{(m+1) \slash 2})$, $f(0)-A g \in \Dom(A^{m \slash 2})$ and
$$ \pd_t^k ( \pd_t f) = \pd_t^{k+1} f \in C^0([0,T];\Dom(A^{(m-k) \slash 2})),\ k=0,\ldots,m, $$
from \eqref{hr2}, we get that $\pd_t^{k+1} v=\pd_t^k w \in C^0([0,T];\Dom(A^{(m+1-k) \slash 2}))$ for $k=0,1,\ldots,m+1$, and the estimate:
\beas
& & \sum_{k=0}^{m+1} \sup_{t \in [0,T]} \| \pd_t^{k+1} v(\cdot,t) \|_{\Dom(A^{(m+1-k) \slash 2})}  \\
& \leq & C \left( \sum_{k=0}^m \| \pd_t^{k+1} f \|_{C^0([0,T];\Dom(A^{(m-k) \slash 2}))} + \| h \|_{\Dom(A^{(m+1) \slash 2})} + \| f(0)- Ag  \|_{\Dom(A^{m \slash 2})} \right).
\eeas
This entails $\pd_t^k v \in C^0([0,T];\Dom(A^{(m+2-k) \slash 2}))$ for $k=1,\ldots,m+2$, and
\bea
& & \sum_{k=1}^{m+2} \sup_{t \in [0,T]} \| \pd_t^{k} v(\cdot,t) \|_{\Dom(A^{(m+2-k) \slash 2})} \nonumber \\
& \leq & C \left( \sum_{k=0}^{m+1} \| \pd_t^{k} f \|_{C^0([0,T];\Dom(A^{(m+1-k) \slash 2}))} + \| h \|_{\Dom(A^{(m+1) \slash 2})} + \| g  \|_{\Dom(A^{(m+2) \slash 2})} \right).
\label{hr3}
\eea
Further, as $Av=f-\pd_t^2 v$ from the first line in \eqref{r0}, we find out that
\bea
\| v(\cdot,t) \|_{\Dom(A^{(m+2) \slash 2})} & \leq & C \left( \| A v(\cdot,t) \|_{\Dom(A^{m \slash 2})} + \|  v \|_{\Dom(A^0)} \right) \nonumber \\
& \leq & C \left( \| f(\cdot,t) \|_{D(A^{m \slash 2})} + \| \pd_t^2 v(\cdot,t) \|_{D(A^{m \slash 2})} + \|  v \|_{\Dom(A^0)} \right). \label{hr4} 
\eea
Here we used the identity $\| v(\cdot,t) \|_{\Dom(A^{(m+2) \slash 2})}^2 = \| A v(\cdot,t) \|_{\Dom(A^{m \slash 2})}^2 + \sum_{k=0}^1 \| A^{k \slash 2} v(\cdot,t) \|_{D(A^0)}^2$ and the estimate $\| A^{1 \slash 2} v(\cdot,t) \|_{D(A^0)}\leq \sum_{k=0}^1 \| A^k v(\cdot,t) \|_{D(A^0)}$.
Since $\|  v \|_{\Dom(A^0)}$ and $\| \pd_t^2 v(\cdot,t) \|_{D(A^{m \slash 2})}$ are majorized by the right hand side of \eqref{hr3}, uniformly in $t \in [0,T]$, \eqref{hr3}-\eqref{hr4} yield the assertion of the theorem for $m+1$.
\end{proof}

\begin{remark}
The result of Theorem \ref{thm-hr} is similar to the one of \cite[Sect. 7.2, Theorem 6]{Ev}, which holds for a bounded domain. This can be seen from the characterization of the $\Dom(A^{k \slash 2})$ for $k=0,\ldots,m$, displayed in Subsection \ref{ss-domain}. Namely, it is worth noticing that the $m^{\rm th}$-order compatibility conditions \cite[Sect. 7.2, Eq. (62)]{Ev} imposed on $f$, $g$ and $h$, are actually hidden in the operatorial formulation of Theorem \ref{thm-hr}.
\end{remark}

\subsection{Characterizing the domain of $A^{m / 2}$ for $m \in \N^*$}
\label{ss-domain}

\subsubsection{Elliptic boundary regularity}
In this subsection we extend the classical elliptic boundary regularity result for the operator $\nabla \cdot c \nabla$ , which is well known in any sufficiently smooth bounded subdomain of $\R^n$ (see e.g. \cite[Sect. 6.3, Theorem 5]{Ev}), to the case of the infinite waveguide $\Omega$ under study. The proof of this result boils down to \cite[Lemma 2.4]{KPS2}
which claims elliptic boundary regularity for the Dirichlet Laplacian in $\Omega$.

\begin{lemma}
\label{lm-ebr}
Let $r$ be a nonnegative integer. We assume that $\pd \omega$ is $C^{r+2}$ and that $c \in W^{r+1,\infty}(\Omega)$ obeys \eqref{c0a}. Then, for any $\varphi \in H^r(\Omega)$, there exists a unique solution $v \in H^{r+2}(\Omega)$ to the boundary problem
\bel{ebr1} 
\left\{ \begin{array}{ll} - \nabla \cdot c \nabla v = \varphi & \mbox{in}\ \Omega \\ v  =  0 & \mbox{on}\ \pd \Omega. \end{array} \right.
\ee
Moreover we have the estimate
\bel{ebr2}
\| v \|_{r+2,\Omega} \leq C_r \| \varphi \|_{r,\Omega},
\ee
where $C_r$ is a positive constant depending only on $r$, $\omega$, the constant $c_m$ appearing in \eqref{c0a} and $\| c \|_{W^{r+1,\infty}(\Omega)}$.
\end{lemma}
\begin{proof}
The proof is by induction on $r$.\\
a) We first consider the case $r=0$.
Due to \eqref{c0a} there is a unique solution $v \in H_0^1(\Omega)$ to \eqref{ebr1} by the Lax-Milgram theorem. Moreover $v$ satisfies the energy estimate 
\bel{ebr3}
\| v \|_{1,\Omega} \leq C \| \varphi \|_{0,\Omega}, 
\ee
where the constant $C>0$ depends only on $\omega$ and $c_m$. Here we used \eqref{c0a} and the Poincar\'e inequality, which holds true in $\Omega$ since $\omega$ is bounded.
Furthermore, $v$ is solution to the boundary value problem
\bel{ebr4} 
\left\{ \begin{array}{ll} - \Delta v = f & \mbox{in}\ \Omega \\ v = 0 & \mbox{on}\ \pd \Omega, \end{array} \right.
\ee
where
\bel{ebr5}
f:=c^{-1} \left( \varphi + \nabla c \cdot \nabla v \right).
\ee
Since $f \in L^2(\Omega)$ then $v \in H^2(\Omega)$ by \cite[Lemma 2.4]{KPS2}, and
$\| v \|_{2,\Omega}$ is upper bounded, up to some multiplicative constant depending only on $\omega$, by $\| f \|_{0,\Omega}$. As a consequence we have
$$ \| v \|_{2,\Omega}  \leq C' \left( \| \varphi \|_{0,\Omega} + \| v \|_{1,\Omega} \right), $$
from \eqref{ebr5}, the constant $C'>0$ depending only on $\omega$, $c_m$ and $\| c \|_{W^{1,\infty}(\Omega)}$. This and \eqref{ebr3} yield \eqref{ebr2} with $r=0$. 

\noindent b) Suppose that the statement of the lemma is true for $r \in \N$ fixed, and assume that $\pd \omega$ is $C^{r+3}$, $c \in W^{r+2,\infty}(\Omega)$ and $\varphi \in H^{r+1}(\Omega)$. Hence the solution $v$ to \eqref{ebr4} belongs to $H^{r+2}(\Omega)$ and satisfies the estimate \eqref{ebr2}, by induction assumption, and we have
$f \in H^{r+1}(\Omega)$ in virtue of \eqref{ebr5}. Further, $v$ being solution to \eqref{ebr4} where the boundary $\pd \omega$ is $C^{r+3}$ then $v \in H^{r+3}(\Omega)$ by \cite[Lemma 2.4]{KPS2}. Moreover $\| v \|_{r+3,\Omega}$ is upper bounded (up to some multiplicative constant depending only on $r$ and $\omega$) by $\| f \|_{r+1,\Omega}$. 
From this and \eqref{ebr5} then follows that
$$ \| v \|_{r+3,\Omega} \leq C'' \left( \| \varphi \|_{r+1,\Omega} + \| v \|_{r+2,\Omega} \right), $$
where the constant $C''>0$ depends only on $r$, $\omega$, $c_m$ and $\| c \|_{W^{r+2,\infty}(\Omega)}$. Putting this together with \eqref{ebr2}, we obtain \eqref{ebr2} where $r$ is replaced by $r+1$, proving that the statement of the lemma remains valid upon substituting $r+1$ for $r$. 
\end{proof}

\subsubsection{The domain of $A^{m / 2}$ for $m \in N^*$}
In this subsection we characterize the domain of 
$A^{m \slash 2}$ for $m \in \N^*$.

\begin{proposition}
\label{pr-dom}
Let $m \in \N^*$ and let $k$ be either $m-1 \slash 2$ or $m$. Assume that $\pd \omega$ is $C^{2k}$ and that $c \in W^{2k-1,\infty}(\Omega)$ fullfills \eqref{c0a}. Then we have
$$
\Dom(A^k) = \{ u \in H^{2k}(\Omega),\ u, Au,\ \ldots, A^{m-1} u \in H_0^1(\Omega) \}.
$$
Moreover, the norm associated with $\Dom(A^k)$ is equivalent to the usual one in $H^{2k}(\Omega)$: we may find a constant $c(k)>1$, depending only on $k$, $\omega$, the constant $c_m$ defined in \eqref{c0a} and $\| c \|_{W^{2k-1,\infty}(\Omega)}$, such that
we have
$$
c(k)^{-1} \| u \|_{\Dom(A^k)} \leq \| u \|_{2k,\Omega} \leq c(k) \| u \|_{\Dom(A^k)},\ u \in D(A^k).
$$
\end{proposition}
\begin{proof}
It suffices to show that
\bel{d3}
\Dom(A^k) \subset \{ u \in H^{2k}(\Omega),\ u, Au,\ \ldots, A^{m-1} u \in H_0^1(\Omega) \},
\ee
and
\bel{d4}
\| u \|_{2k,\Omega} \leq c(k) \| u \|_{\Dom(A^k)},\ u \in D(A^k).
\ee
The proof is by induction on $m$.

\noindent a) We start with $m=1$ and notice from the very definition of $A^{1 \slash 2}$ that $\Dom(A^{1 \slash 2}) = \Dom (q_A) = H_0^1(\Omega)$. Moreover we have
$$ \| A^{1 \slash 2} u \|_{0,\Omega}^2 = q_A[u] \geq c_m \| \nabla u \|_{0,\Omega}^2,\ u  \in \Dom(A^{1 \slash 2}), $$
in virtue of \eqref{c0a}.
Bearing in mind that $c_m \in (0,1)$, we obtain that $\| u \|_{1,\Omega} \leq c_m^{-1 \slash 2} \| u \|_{\Dom(A^{1 \slash 2})}$ for every $u \in \Dom(A^{1 \slash 2})$. This establishes \eqref{d4} for $k=1 \slash 2$.

Similarly, bearing in mind that $\Dom(A) = \{ u \in H_0^1(\Omega),\ A u \in L^2(\Omega) \}$, we apply Lemma \ref{lm-ebr} with $r=0$ and $\varphi = A u$, where $u \in \Dom(A)$ is arbitrary. We find that $u \in H^2(\Omega)$ satisfies $\| u \|_{2,\Omega} \leq C_0 \| A u \|_{0,\Omega}$, which entails \eqref{d3}-\eqref{d4} for $k=1$.

\noindent b) Let us now suppose that the statement of the lemma is true for some $m \in \N^*$ fixed. Pick $k \in \{ m-1 \slash 2, m \}$ and assume that $\pd \omega$ is $C^{2(k+1)}$ and that
$c \in W^{2k+1,\infty}(\Omega)$ satisfies \eqref{c0a}. 
As $\Dom(A^{k+1})= \{ u \in \Dom(A^{k}),\ A u \in \Dom(A^{k}) \}$, we deduce from the induction assumption that
we have 
$$ \Dom(A^{k+1})= \{ u \in H^{2k}(\Omega),\ A u \in H^{2k}(\Omega)\ \mbox{and}\ u,\ldots, A^m u \in H_0^1(\Omega) \}, $$
with $\| A^j u \|_{2k,\Omega} \leq c(k) \| A^j u \|_{D(A^k)}$ for $j=0,1$.
Thus, applying Lemma \ref{lm-ebr}, with $r=2k$ and $\varphi = Au$, for $u \in \Dom(A^{k+1})$, we get that $u \in H^{2(k+1)}(\Omega)$, proving 
\eqref{d3} where $(k+1,m+1)$ is substituted for $(k,m)$. Moreover, it holds true that
$$
\| u \|_{2(k+1),\Omega} \leq C_{2k} \| A u \|_{2k,\Omega},
$$
and since $Au \in D(A^k)$, the induction assumption yields
$\| A u \|_{2k,\Omega} \leq c(k) \| A u \|_{\Dom(A^k)}$. Therefore $\| u \|_{2(k+1),\Omega}$ is majorized, up to a multiplicative constant depending only on $\omega$, $c$ and $m$, by $\| u \|_{\Dom(A^{k+1})}$, which is \eqref{d4} where $k+1$ is substituted for $k$.
\end{proof}

In view of Theorem \ref{thm-hr} and Proposition \ref{pr-dom} we obtain the following result.

\begin{follow}
\label{cor-eur}
Let $m$ be a natural number. Assume that $\pd \omega$ is $C^{m+1}$, that $c \in W^{m,\infty}(\Omega)$ fulfills \eqref{c0a} and
that $(\theta_0,\theta_1) \in \Dom(A^{(m+1) \slash 2}) \times \Dom(A^{m \slash 2})$. Then the initial boundary value problem \eqref{S1} admits a unique solution 
$$u \in \bigcap_{k=0}^{m+1} C^k([0,T]; H^{m+1-k}(\Omega)). $$
Moreover, we have
\bel{d4b}
\sum_{k=0}^{m+1} \| u \|_{C^k([0,T];H^{m+1-k}(\Omega))} \leq C \left( \| \theta_0 \|_{m+1,\Omega} + \| \theta_1 \|_{m,\Omega} \right),
\ee
for some constant $C>0$ depending only on $T$, $\omega$ and $\| c \|_{W^{m,\infty}(\Omega)}$.
\end{follow}


\section{Global Carleman estimate for hyperbolic equations in cylindrical domains}
\label{sec-Carleman}
\setcounter{equation}{0}
In this section we establish a global Carleman estimate for the system \eqref{S1}. To this purpose we start by time-symmetrizing the solution $u$ of \eqref{S1}. Namely, we put
\bel{t-sym}
u(x,t) := u(x,-t),\ x \in \Omega,\ t \in (-T,0).
\ee
Under the conditions of Theorem \ref{T.1}, and since 
$\theta_1 = 0$, it is not hard to check that 
$$u \in \bigcap_{k=3}^4 C^k([-T,T];H^{5-k}(\Omega)). $$
With a slight abuse of notations we put $Q := \Omega \times (-T,T)$, $\Sigma := \Gamma \times (-T,T)$ and $Q_L := \Omega_L \times (-T,T)$, $\Sigma_L := \partial \omega \times (-L,L) \times (-T,T)$ for any $L>0$, in the remaining part of this text.

\subsection{The case of second order hyperbolic operators}
In view of establishing a Carleman estimate for the operator
\bel{c1b}
A:=A(x,t,\pd)=\partial_t^2 - \nabla \cdot c \nabla + R, 
\ee
where $R$ is a first-order partial differential operator with $L^\infty(Q)$ coefficients, we define for every $\delta>0$ and $\gamma>0$ the following weight functions:
\bel{c3}
\psi(x,t)=\psi_\delta(x,t):=| x'- \delta a' |^2-x_n^2-t^2\ \mbox{and}\
\varphi(x,t)=\varphi_{\delta,\gamma}(x,t):={\rm e}^{\gamma \psi(x,t)},\ (x,t) \in Q.
\ee
Further for $L>0$ and $T>0$ we introduce the space
\bel{X_0^2}
\mathcal{X}_{L,T} := \left\{ u \in H^2(Q_{L});\ \partial_{x_n}^j u(\cdot, \pm L, \cdot)=0\ \mbox{in}\ \omega \times (-T, T),\ \partial_{t}^j u(\cdot, \pm T)=0\ \mbox{in}\
\Omega_L,\ j=0,1 \right\}, 
\ee
and state the:

\begin{proposition}
\label{pr-ic}
Let $A$ be defined by \eqref{c1b}, where $c$ obeys \eqref{c0a}--\eqref{c1}, and let $\ell$ be positive. Then there exist $\delta_0>0$ and $\gamma_0>0$ such that for all 
$\delta \geq \delta_0$ and $\gamma \geq \gamma_0$, we may find $L>\ell$, $T>0$ and $s_0>0$ for which the estimate
\bel{c16}
s \sum_{j=0,1} s^{2(1-j)} \| {\rm e}^{s \varphi} \nabla_{x,t}^j v \|_{0,Q_L}^2
\leq C \left( \| {\rm e}^{s\varphi} A v \|_{0,Q_L}^2 + s \sum_{j=0,1} s^{2(1-j)} \| {\rm e}^{s\varphi} \nabla_{x,t}^j v \|_{0,\pd Q_L}^2 \right),
\ee
holds for any $s \geq s_0$ and $v \in \mathcal{X}_{L,T}$. Here $C$ is a positive constant depending only on $\omega$, $a'$, $\mathfrak{a}_0$, $\delta_0$, $\gamma_0$, $s_0$, $c_m$ and $c_M$.

Moreover there exists a constant $d_\ell>0$, depending only on $\omega$, $\ell$, $\delta_0$ and $\gamma_0$, such that the weight function $\varphi$ defined by \eqref{c3} satisfies
\bel{d5}
\varphi(x', x_n, 0) \geq d_\ell,\ (x', x_n) \in \overline{\omega} \times [-\ell, \ell],
\ee
and we may find $\eps \in (0, (L-\ell) \slash 2)$ and $\nu >0$ so small that we have:
\bea
\max_{x \in\overline{\omega}\times [-L, L]}
\varphi(x', x_n, t) & \leq & \tilde{d}_\ell:=d_\ell {\rm e}^{-\gamma  \nu^2},\ |t|\in [T-2\eps, T], \label{d6} \\
\max_{(x', t) \in \overline{\omega} \times [-T, T]} \varphi(x', x_n, t) & \leq & \tilde{d}_\ell,\ |x_n|\in [L-2\eps, L]. \label{d7}
\eea
\end{proposition}

\begin{proof}
The proof is divided in three parts and essentially boils down to \cite[theorem 3.2.1']{I}.\\

\noindent {\it First part: Definition of $\delta_0$, $L$ and $T$.}
Bearing in mind that $| x' - \delta a' |^2 - | y' - \delta a' |^2 = | x' | - |y'|^2 - 2 \delta a' \cdot (x' - y')$ for all $x', y' \in \omega$, we see that
$\sup_{x' \in \omega} |x' - \delta a' |^2 - \inf_{x' \in \omega} |x' - \delta a' |^2 \leq | \omega | ( | \omega | + 4 \delta |a'|)$ for every $\delta >0$, where
$|\omega|:=\sup_{x'\in \overline{\omega}} |x'|$. Hence the function
\bel{d1}
g_\ell(\delta) = \left( \sup_{x' \in \omega} |x' - \delta a' |^2 - \inf_{x' \in \omega} |x' - \delta a' |^2 + \ell^2 \right)^{1 \slash 2}
\ee
scales at most like $\delta^{1 \slash 2}$, proving that there exists $\delta_0>0$ so large that
\bel{d2}
\delta \mathfrak{a}_0 >  \left( \left( 1 + \frac{2}{{c_m}^{1 \slash 2}} \right) g_\ell(\delta) + | \omega |+ 2 \right) c_M + 2,\ \delta \geq \delta_0.
\ee
Further, since $\omega$ is bounded and $a' \neq 0_{\R^{n-1}}$ by \eqref{c1}, we may as well assume upon possibly enlarging $\delta_0$, that we have in addition
$c_m^{1 \slash 2} \inf_{x' \in \overline{\omega}} | x' - \delta a' | > g_\ell(\delta)$ for all $\delta \geq \delta_0$.
This and \eqref{d2} yield that there exists $\vartheta>0$ so small that the two following inequalities
\bel{c2}
\delta \mathfrak{a}_0 - \left( L+ | \omega | + 2 \left( 1 + \frac{T}{{c_m}^{1 \slash 2}} \right) \right) c_M - 2 >0,
\ee
and
\bel{c4} 
c_m^{1 \slash 2} \inf_{x' \in \overline{\omega}} | x' - \delta a' | > T,
\ee
hold simultaneously for every $L$ and $T$ in $(g_\ell(\delta), g_\ell(\delta)+\vartheta)$, uniformly in $\delta \geq \delta_0$.\\

\noindent {\it Second part: Proof of \eqref{c16}.}
We first introduce the following notations, we shall use in the remaining part of the proof. For notational simplicity we put $\bx:=(x,t)$ for $(x,t) \in Q_L$ and $\nabla_{\bx}=(\pd_1,\ldots,\pd_{n},\pd_{n+1})^T$. We also write $\xi'=(\xi_1,\ldots,\xi_{n-1})^T \in \R^{n-1}$, $\xi=(\xi_1,\ldots,\xi_n)^T \in \R^n$ and $\txi=(\xi_1,\ldots,\xi_n,\xi_{n+1})^T \in \R^{n+1}$. We call $A_2$ the principal part of the operator $A$, that is
$A_2=A_2(\bx,\pd)=\pd_t^2- c(x) \Delta$, and denote its symbol by $A_2(x,\txi)=c(x) | \xi |^2 - \xi_{n+1}^2$, where $| \xi |= \left( \sum_{j=1}^n \xi_j^2 \right)^{1 \slash 2}$.
Since $A_2(x,\nabla_{\bx} \psi(\bx)) = 4 \left( c(x) ( | x' - \delta a' |^2 + x_n^2 ) - x_{n+1}^2 \right)$ for every $\bx \in \overline{Q}_L$, we have
\bel{c10}
A_2(x,\nabla_{\bx} \psi(\bx))>0,\ \bx \in \overline{Q}_L,
\ee 
by \eqref{c1} and \eqref{c4}.
For all $\bx \in \overline{Q}_L$ and $\txi \in \R^{n+1}$, put
\bel{c11}
J(\bx,\txi) = J =\sum_{j,k=1}^{n+1} \frac{\pd A_2}{\pd \xi_j} \frac{\pd A_2}{\pd \xi_k} \pd_j \pd_k \psi + 
\sum_{j,k=1}^{n+1} \left( \left( \pd_k \frac{\pd A_2}{\pd \xi_j} \right) \frac{\pd A_2}{\pd \xi_k} - (\pd_k A_2) \frac{\pd^2 A_2}{\pd \xi_j \pd \xi_k}
\right) \pd_j \psi,
\ee
where, for the sake of shortness, we write $\partial_j$, $j \in \N_{n+1}^*$, instead of $\pd \slash \pd x_j$, and $x_{n+1}$ stands for $t$.
Assuming that
\bel{c12}
A_2(x,\txi)=c(x) | \xi |^2 - \xi_{n+1}^2 = 0,\ x \in \overline{\Omega},\ \txi \in \R^{n+1} \backslash \{ 0 \},
\ee
and
\bel{c13}
\nabla_{\txi} A_2(x,\txi) \cdot \nabla_{\bx} \psi(\bx) = 4 \left[ c(x) ( \xi' \cdot (x' - \delta a') - \xi_n x_n ) + \xi_{n+1} x_{n+1} \right]=0,\ 
\bx \in \overline{Q}_L,\ \txi \in \R^{n+1} \backslash \{ 0 \},
\ee
we shall prove that $J(\bx,\txi)>0$ for any $(\bx,\txi) \in \overline{Q}_L \times \R^{n+1}$. To this end we notice that the first sum in the rhs of \eqref{c11} reads
$\langle \mbox{Hess}(\psi) \nabla_{\txi} A_2 , \nabla_{\txi} A_2 \rangle = 8 \left( c^2 ( | \xi' |^2 - \xi_n^2 ) - \xi_{n+1}^2 \right)$, 
and that
\beas
& & \sum_{j,k=1}^{n+1} \left( \left( \pd_k \frac{\pd A_2}{\pd \xi_j} \right) \frac{\pd A_2}{\pd \xi_k} - (\pd_k A) \frac{\pd^2 A_2}{\pd \xi_j \pd \xi_k}
\right) \pd_j \psi = \sum_{j,k=1}^{n} \left( \left( \pd_k \frac{\pd A_2}{\pd \xi_j} \right) \frac{\pd A_2}{\pd \xi_k} - (\pd_k A_2) \frac{\pd^2 A_2}{\pd \xi_j \pd \xi_k}
\right) \pd_j \psi \\
& = &  2 c\sum_{j,k=1}^n \left( \left( 2 \xi_j \xi_k  - | \xi|^2 \frac{\pd \xi_j}{\pd \xi_k} \right) \pd_k c \right)  \pd_j \psi = 2 \left( 2 (\nabla c \cdot \xi) (\nabla \psi \cdot \xi)  -  (\nabla c \cdot \nabla \psi) | \xi |^2 \right),
\eeas
since $\left(\pd_k \frac{\pd A_2}{\pd \xi_j} \right) \frac{\pd A_2}{\pd \xi_k} - (\pd_k A_2) \frac{\pd^2 A_2}{\pd \xi_j \pd \xi_k}=0$ if either $j$ or $k$ is equal to $n+1$.
Therefore we have
\beas
J & = & 2 \left[ 4 \left( c^2 ( | \xi' |^2 - \xi_n^2 ) - \xi_{n+1}^2 \right) + 2 c (\nabla c \cdot \xi) (\nabla \psi \cdot \xi)  - c (\nabla c \cdot \nabla \psi) | \xi |^2  \right] \\
& = & 4 \left[ 2 c^2 (|\xi'|^2 - \xi_{n}^2) - \left( 2 + (x'-\delta a') \cdot \nabla_{x'} c  - x_n \pd_n c \right) \xi_{n+1}^2 - 2 x_{n+1} \xi_{n+1} \nabla c \cdot \xi \right],
\eeas
from \eqref{c12}-\eqref{c13}. Further, in view of \eqref{c12} we have $c^2 ( | \xi'|^2 - \xi_n^2 ) \geq - c ^2 | \xi |^2 \geq -c \xi_{n+1}^2$ and
$| \nabla c \cdot \xi | \leq | \nabla c | | \xi |  \leq ( | \nabla c | \slash c^{1 \slash 2} )| \xi_{n+1} |$, whence
\bel{c15}
J \geq 4 \left[ \delta  a' \cdot \nabla_{x'} c -  \left( x \cdot \nabla c + 2 c + 2 T \frac{| \nabla c |}{c^{1 \slash 2}} + 2 \right) \right] \xi_{n+1}^2.
\ee
Here we used that fact that $x_{n+1}=t \in [0,T]$. Due to \eqref{c0b}-\eqref{c1}, the rhs of \eqref{c15} is lower bounded, up to the multiplicative constant $4 \xi_{n+1}^2$, by the lhs of \eqref{c2}. Since $\xi_{n+1}$ is non zero by \eqref{c0b} and \eqref{c12}, then we obtain $J(\bx,\txi)>0$ for all $(\bx,\txi) \in \overline{Q}_L \times \R^{n+1}$. With reference to
\eqref{c10}, we may apply \cite[Theorem 3.2.1']{I}, getting two constants $s_0=s_0(\gamma)>0$ and $C>0$ such that \eqref{c16}
holds for any $s \geq s_0$ and $v \in H^2(Q_L)$.\\

\noindent {\it Third part: Proof of \eqref{d5}--\eqref{d7}.}
First, \eqref{d5} follows readily from \eqref{c3}, with
$d_\ell:={\rm e}^{\gamma \beta_\ell}$ and $\beta_\ell:=\inf_{x' \in \omega} \left|x'-\delta a' \right|^2 - \ell^2$. Next,  for $\nu \in (0,\vartheta)$ arbitrarily fixed, we put 
\bel{condition}
L=T=g_\ell(\delta)+ \nu.
\ee
Notice for further reference from \eqref{d1}, \eqref{c4} and \eqref{condition}, that we have
\bel{d10}
\beta_\ell \geq \frac{g_\ell(\delta)^2}{c_m} - \ell^2 \geq \left( \frac{1-c_m}{c_m} \right) \ell^2 >0,
\ee
since $c_m \in (0,1)$, by assumption.
Similarly, as $T^2 \geq g_\ell(\delta)^2 + \nu^2=\sup_{x' \in \omega} \left| x' - \delta a' \right|^2 - ( \beta_{\ell}^2 -  \nu^2)$, we dedude from \eqref{c3} that
$$
 \varphi(x', x_n, \pm T)\leq {\rm e}^{\gamma \left(\sup_{x'\in\omega} \left|x'-\delta a' \right|^2-x_n^2-T^2 \right)}
\leq {\rm e}^{\gamma (\beta_\ell-\nu^2)}{\rm e}^{-\gamma x_n^2},\ (x', x_n) \in \overline{\omega}\times [-L, L].
$$
With reference to \eqref{d10} we may thus chose $\eps \in (0, (L-\ell)/2)$ so small that
$$
\varphi(x', x_n, t) \leq d_\ell {\rm e}^{-\gamma  \nu ^2} {\rm e}^{-\gamma x_n^2},\
 (x', x_n) \in \overline{\omega} \times [-L, L],\ |t|\in [T-2 \eps, T],
$$
which entails \eqref{d6}. Finally, since $t$ and $x_n$ play symmetric roles in \eqref{c3}, and since $T=L$,
we obtain \eqref{d7} by substituting $(T,t)$ for $(L,x_n)$ in \eqref{d6}.
\end{proof}

\subsection{A Carleman estimate for the system \eqref{S1}}
 \label{sec-SCE}
In this subsection we derive from Proposition \ref{pr-ic} a global Carleman estimate for the solution to the boundary value problem
\bel{c18}
\left\{  \begin{array}{ll} \pd_t^2 u - \nabla \cdot c \nabla u = f & \mbox{in}\ Q \\ u= 0 & \mbox{on}\ \Sigma,
\end{array} \right.
\ee
where $f \in L^2(Q)$. 
To this purpose we introduce a cut-off function $\chi \in C^2(\R;[0,1])$, such that
\bel{c21}
\chi(x_n):= \left\{ \begin{array}{cl} 1 & \mbox{if}\ |x_n| < L - 2 \eps, \\ 0 & \mbox{if}\ |x_n| \geq L - \eps, \end{array} \right.
\ee
where $\eps$ is the same as in Proposition \ref{pr-ic}, and we set
$$u_\chi(x,t):=\chi(x) u(x,t)\ \mbox{and}\ f_\chi(x,t):=\chi(x) f(x,t),\ (x,t) \in Q. $$
 
\begin{follow}
\label{cor-ec}
 
Let $f \in L^2(Q)$. Then, under the conditions of Proposition \ref{pr-ic},
there exist two constants $s_*>0$ and $C>0$, depending only on $\omega$, $\ell$, $M_0$, $\eta_0$, $a'$, $\mathfrak{a}_0$, $c_m$ and $c_M$, such that the estimate
$$
s \sum_{j=0,1} s^{2(1-j)} \| {\rm e}^{s \varphi} \nabla_{x,t}^j u \|_{0,Q_L}^2 
\leq C \left(  \| {\rm e}^{s \varphi} f \|_{0,Q_L}^2 + s^3 {\rm e}^{2 s \tilde{d}_\ell}  \| u \|_{1,Q_L}^2  + s \sum_{j=0,1} s^{2(1-j)} \| {\rm e}^{s \varphi} \nabla_{x,t}^j u_\chi \|_{0,\Sigma_L}^2 \right),
$$
holds for any solution $u$ to \eqref{c18}, uniformly in $s \geq s_*$.
\end{follow}
\begin{proof}
Since $u$ is solution to \eqref{c18} we have
$$
\pd_t^2 u_{\chi} - \nabla \cdot c \nabla u_\chi = f_\chi + R_1 u\ \mbox{in}\ Q,
$$
where 
\bel{c22b}
R_1=R_1(x,\partial):=[\chi, \nabla \cdot c \nabla] = - (c \Delta \chi+ \nabla c \cdot \nabla \chi + 2 c \nabla \chi \cdot \nabla),
\ee
is a first-order differential operator. Therefore, the function $v(x,t):=\eta(t) u_\chi(x,t)$, where $\eta \in C^2(\R;[0,1])$ is such that
$$ \eta(t):= \left\{ \begin{array}{cl} 1 & \mbox{if}\ |t| < T - 2 \eps, \\ 0 & \mbox{if}\ |t| \geq T - \eps, \end{array} \right. $$
satisfies
$$
\pd_t^2 v -  \nabla \cdot c \nabla v = g := \eta f_\chi + \eta R_1 u + \eta'' u_\chi + 2 \eta' \pd_t u_\chi\ \mbox{in}\ Q. 
$$
Moreover, as $v(\cdot,\pm L,\cdot)=\pd_{x_n} v(\cdot,\pm L,\cdot)=0$ in $\omega \times (-T,T)$ and $v(\cdot,\pm T)=\pd_t v(\cdot,\pm T)=0$ in $\Omega_L$, we may apply Proposition \ref{pr-ic}, getting
\bel{c24}
s \sum_{j=0,1} s^{2(1-j)}  \| {\rm e}^{s \varphi} \nabla_{x,t}^j v \|_{0,Q_L}^2 
\leq C \left( \| {\rm e}^{s \varphi} g \|_{0,Q_L}^2 + s \sum_{j=0,1} s^{2(1-j)} \| {\rm e}^{s \varphi} \nabla_{x,t}^j v \|_{0,\pd Q_L}^2 \right).
\ee
Further, bearing in mind that $\pd Q_L= \Sigma_L \cup \left( \omega \times \{ \pm L \} \times (-T,T) \right) \cup \left( \Omega_L \times \{ \pm T \} \right)$, we deduce from
the vanishing of $v(\cdot,\pm L,\cdot)$ and $\nabla_{x,t} v(\cdot,\pm L,\cdot)$ in $\omega \times (-T,T)$, and the one of $v(\cdot,\pm T)$ and $\nabla_{x,t} v(\cdot,\pm T)$ in $\Omega_L$, that
\bel{c25}
\| {\rm e}^{s \varphi} \nabla_{x,t}^j v \|_{0,\pd Q_L}= \|  {\rm e}^{s \varphi} \nabla_{x,t}^j v \|_{0,\Sigma_L},\ j=0,1.
\ee
Next we know from \eqref{d7} and \eqref{c22b} that
\bel{c26}
\| {\rm e}^{s \varphi} \eta R_1 u \|_{0,Q_L} \leq C {\rm e}^{s \tilde{d}_\ell} \| u \|_{L^2(-T,T;H^1(\Omega_L))},
\ee
and from \eqref{d6} that
\bel{c26bis}
\| {\rm e}^{s \varphi} (\eta'' u_\chi + 2 \eta' \pd_t u_\chi) \|_{0,Q_L} \leq C {\rm e}^{s \tilde{d}_\ell} \| u_\chi \|_{H^1(-T,T;L^2(\Omega_L))}.
\ee
Hence, putting \eqref{c24}--\eqref{c26bis} together, we find out that
\bel{c24bis}
s \sum_{j=0,1} s^{2(1-j)} \| {\rm e}^{s \varphi} \nabla_{x,t}^j v \|_{0,Q_L}^2 \leq C \left(\| {\rm e}^{s \varphi} f \|_{0,Q_L}^2 +  s \sum_{j=0,1} s^{2(1-j)} \| {\rm e}^{s \varphi} \nabla_{x,t}^j u_\chi \|_{0,\Sigma_L}^2 + {\rm e}^{2s \tilde{d}_\ell} \| u \|_{H^1(Q_L)}^2 \right).
\ee
The next step of the proof involves noticing from \eqref{d6} that $\| {\rm e}^{s \varphi} (1-\eta) \nabla_x^j u_\chi \|_{0,Q_L} \leq {\rm e}^{s \tilde{d}_\ell} \| \nabla_x^j u_\chi \|_{0,Q_L}$ for $j=0,1$, hence
\bea
\| {\rm e}^{s \varphi} \nabla_x^j u_\chi \|_{0,Q_L} & \leq & \| {\rm e}^{s \varphi} (1-\eta) \nabla_x^j u_\chi \|_{0,Q_L} + \| {\rm e}^{s \varphi} \nabla_x^j v \|_{0,Q_L} \nonumber \\
& \leq & {\rm e}^{s \tilde{d}_\ell} \| \nabla_x^j u_\chi \|_{0,Q_L} + \| {\rm e}^{s \varphi} \nabla_x^j v \|_{0,Q_L},\ j=0,1.
\label{c27}
\eea
Furthermore, by combining the identity
$\pd_t u_\chi=(1-\eta) \pd_t u_\chi + \eta \pd_t u_\chi=(1-\eta) \pd_t u_\chi - \eta' \pd_t u_\chi + \pd_t v$ with \eqref{d6}, we get that
\beas
\| {\rm e}^{s \varphi} \pd_t u_\chi \|_{0,Q_L} & \leq & \| {\rm e}^{s \varphi} (1-\eta)  \pd_t u_\chi \|_{0,Q_L} + \| {\rm e}^{s \varphi} \dot{\eta} u_\chi \|_{0,Q_L} + \| {\rm e}^{s \varphi} \pd_t v \|_{0,Q_L} \nonumber \\
& \leq & {\rm e}^{s \tilde{d}_\ell} \left( \| \pd_t u_\chi \|_{0,Q_L} + \| \eta' \|_{L^{\infty}(-T,T)}  \|  u_\chi  \|_{0,Q_L} \right)+ \| {\rm e}^{s \varphi} \pd_t v \|_{0,Q_L},
\eeas
which, together with \eqref{c27}, yields
\bel{c29}
\sum_{j=0,1} s^{2(1-j)} \| {\rm e}^{s \varphi} \nabla_{x,t}^j u_\chi \|_{0,Q_L}^2 \leq C \sum_{j=0,1} s^{2(1-j)} \left( {\rm e}^{2s \tilde{d}_\ell}  \| \nabla_{x,t}^j u_\chi \|_{0,Q_L}^2
+ \| {\rm e}^{s \varphi} \nabla_{x,t}^j v \|_{0,Q_L}^2 \right).
\ee
Similarly, using \eqref{d7}, we derive from the identity $\pd_t^j u = \pd_t^j u_\chi + (1-\chi) \pd_t^j u$ for $j=0,1$, that
$$
\| {\rm e}^{s \varphi} \pd_t^j u \|_{0,Q_L} \leq \| {\rm e}^{s \varphi} \pd_t^j u_\chi \|_{0,Q_L} + {\rm e}^{s \tilde{d}_\ell} \| \pd_t^j u \|_{0,Q_L},\ j=0,1,
$$
and from $\nabla_x u = \nabla_x u_\chi + (1-\chi) \nabla_x u - \chi' (0, \ldots, 0, u)^T$, that
$$
\| {\rm e}^{s \varphi} \nabla_x u \|_{0,Q_L} \leq \| {\rm e}^{s \varphi} \nabla_x u_\chi \|_{0,Q_L} + {\rm e}^{s \tilde{d}_\ell} \left( \| \nabla_x u \|_{0,Q_L} + \| \chi' \|_{L^{\infty}(-L,L)} \| u \|_{0,Q_L} \right).
$$
As a consequence we have
\bel{c30}
\sum_{j=0,1} s^{2(1-j)} \| {\rm e}^{s \varphi} \nabla_{x,t}^j u \|_{0,Q_L}^2 
\leq C \sum_{j=0,1} s^{2(1-j)} \left( {\rm e}^{2s \tilde{d}_\ell} \| \nabla_{x,t}^j u  \|_{0,Q_L}^2 + \| {\rm e}^{s \varphi} \nabla_{x,t}^j u_\chi \|_{0,Q_L}^2 \right).
\ee
Finally we obtain the desired result by gathering \eqref{c24bis} and \eqref{c29}-\eqref{c30}.
\end{proof}

\section{Inverse problem}
\label{sec-inverse}
\setcounter{equation}{0}

\subsection{Linearized inverse problem and preliminary estimate}
\label{sec-lip}

In this subsection we introduce the linearized inverse problem associated with \eqref{S1} and relate the first Sobolev norm of the conductivity to some suitable initial condition of this boundary problem.

Namely, given $c_i \in \Lambda_{\Gamma}$ for $i=1,2$, we note $u_i$ the solution to \eqref{S1} where $c_i$ is substituted for $c$, suitably extended to $(-T,0)$ in accordance with \eqref{t-sym}. Thus, putting
\bel{s0}
c:=c_1-c_2\ \mbox{and}\ f_c:=\nabla \cdot ( c \nabla u_2 ),
\ee
it is clear from \eqref{S1} that the function $u:=u_1-u_2$ is solution to the linearized system
\bel{s1}
\left\{  \begin{array}{ll} \pd_t^2 u - \nabla \cdot ( c_1 \nabla u )  =  f_c & \mbox{in}\ Q \\ u  =  0 & \mbox{on}\ \Sigma \\
u(\cdot,0) = \pd_t u(\cdot,0)  =  0 & \mbox{in}\ \Omega.
\end{array} \right.
\ee
By differentiating $k$-times \eqref{s1} with respect to $t$, for $k \in \N^*$ fixed, we see that $u^{(k)}:=\pd_t^k u$ is solution to
\bel{s2}
\left\{  \begin{array}{ll} \pd_t^2 u^{(k)} - \nabla \cdot ( c_1 \nabla u^{(k)} )  =  f_c^{(k)} & \mbox{in}\ Q \\ u^{(k)}  = 0 & \mbox{on}\ \Sigma, \\
\end{array} \right.
\ee
with $f_c^{(k)}:=\pd_t^k f_c=\nabla \cdot ( c \nabla u_2^{(k)} )$, where $u_2^{(k)}$ stands for $\pd_t^k u_2$. 

We stick with the notations of Corollary \ref{cor-ec}. In particular, for any function $v$, we denote $\chi v$ by $v_{\chi}$, where $\chi$ is defined in \eqref{c21}.
Upon multiplying both sides of the identity \eqref{s2} by $\chi$, we obtain that 
\bel{s4}
\left\{  \begin{array}{ll} \pd_t^2 u_{\chi}^{(k)} - \nabla \cdot ( c_1 \nabla u_{\chi}^{(k)} ) =  f_{c_\chi}^{(k)}- g_k & \mbox{in}\ Q \\ u_\chi^{(k)} =  0 & \mbox{on}\ \Sigma, \\
\end{array} \right.
\ee
with
\bel{s4b}
f_{c_\chi} := \nabla \cdot (c_\chi \nabla u_2)\ \mbox{and}\
g_k:= \nabla \cdot (c_1 (\nabla \chi) u^{(k)}) + c_1 \nabla \chi \cdot \nabla u^{(k)} + c  \nabla \chi \cdot \nabla u_2^{(k)}.
\ee
Notice that $g_k$ is supported in 
$\tilde{\Omega}_{\eps}:= \{ x=(x',x_n),\ x' \in \omega\ \mbox{and}\ \abs{x_n} \in (L-2 \eps, L-\eps) \}$.

Having said that we may now upper bound, up to suitable additive and multiplicative constants, the $e^{s \varphi(\cdot,0)}$-weighted first Sobolev norm of the conductivity $c_\chi$ in $\Omega_L$, by the corresponding norm of the initial condition $u_\chi^{(2)}(\cdot,0)$.

\begin{lemma}
\label{lm1}
Let $u$ be the solution to the linearized problem \eqref{s1} and let $\chi$ be defined by \eqref{c21}. Then there exist two constants $s_*>0$ 
and $C>0$, depending only on $\omega$, $\varepsilon$ and the constant $M_0$ defined by \eqref{S100}, such that the estimate
$$ \sum_{j=0,1} \| e^{s \varphi(\cdot,0)} \nabla^j c_{\chi} \|_{0,\Omega_L}^2 \leq 
C s^{-1} \left( \sum_{j=0,1} \| e^{s \varphi(\cdot,0)} \nabla^j u_{\chi}^{(2)}(\cdot,0) \|_{0,\Omega_L}^2  + e^{2 s \tilde{d}_\ell} \right),
$$
holds for all $s \geq s_*$.
\end{lemma}
\begin{proof}
Let $\Omega_*$ be an open subset of $\R^n$ with $C^2$ boundary, such that
\bel{s5}
\omega_* \times (-(L-\eps) , L - \eps) \subset \Omega_* \subset \omega_* \times (-L , L),
\ee
where $\eps$ is defined by Proposition \ref{pr-ic}.
We notice from \eqref{c21} and \eqref{ic0} that $\pd_i^j c_{\chi} \in H_0^1(\Omega_*)$ for all $i \in \N_n^*$ and $j=0,1$. 

Further, with reference to \eqref{S10} we may assume upon possibly enlarging $\delta \in [\delta_0,+\infty)$, where $\delta_0$ is the same as in Proposition \ref{pr-ic}, that we have
$$
| \nabla \theta_0 \cdot (x_1 - \delta a_1 , \ldots , x_{n-1} - \delta a_{n-1} , - x_n ) | \geq \mu_0 >0,\ x \in \Omega_*.
$$
Thus applying \cite[Proposition  2.2]{IY03} \footnote{Let $D$ be a bounded open subset of $\R^n$, $n \geq 1$, with $C^2$ boundary, and consider the first-order operator 
$P(x,\pd) := \sum_{i=1}^n p_i(x) \pd_i + p_0(x)$, 
where $p_0 \in C^0(\overline{D})$ and $p:=(p_1,\ldots,p_n) \in C^1(\overline{D})^n$. Assume that
$$
| p(x) \cdot (x_1-\delta a_1, \ldots, x_{n-1}-\delta a_{n-1}, -x_n) | \geq \mathfrak{p}_m,\ x \in \overline{D},
$$
for some $\mathfrak{p}_m>0$. Then for any $\mathfrak{p}_M \geq \max \{\| p_0 \|_{C^0(\overline{D})}, \| p_i \|_{C^1(\overline{D})},\ i \in \N_n^* \}$, there exist $s_*>0$ and $C>0$, depending only on $\mathfrak{p}_M$, such that the estimate 
$$ \| e^{s \varphi(\cdot,0)} v \|_{0,D}^2 \leq C s^{-1}\| e^{s \varphi(\cdot,0)} P v \|_{0,D}^2, $$
holds for all $s \geq s_*$ and $v \in H_0^1(D)$.}
with $D=\Omega_*$, $P(x,\pd) v=\nabla \cdot ( (\nabla \theta_0) v )$ and $v=\pd_i^j c_{\chi} \in H_0^1(\Omega_*)$ since $\chi(x_n)=0$ for $x_n \geq L- \eps$, for $i \in \N_n^*$ and $j=0,1$, we obtain that
\bel{s5b}
s \| e^{s \varphi(\cdot,0)} \pd_i^j c_{\chi} \|_{0,\Omega_*}^2 \leq C \| e^{s \varphi(\cdot,0)} \nabla \cdot ( ( \pd_i^j  c_{\chi} ) \nabla \theta_0 ) \|_{0,\Omega_*}^2,\ i \in \N_n^*.\ j=0,1.
\ee
Since $c_{\chi}(x',x_n)=0$ a.e. in $\omega_* \times \left( (-L,-(L-\eps)) \cup (L- \eps, L) \right)$ by \eqref{c21}, we have \\
$\| e^{s \varphi(\cdot,0)} \pd_i^j c_{\chi} \|_{0,\Omega_*}=\| e^{s \varphi(\cdot,0)} \pd_i^j c_{\chi} \|_{0,\Omega_L}$ for each $i \in \N_n^*$ and $j=0,1$, from \eqref{s5}. We derive from this and \eqref{s5b} that
\bel{s6}
s \| e^{s \varphi(\cdot,0)} \pd_i^j c_{\chi} \|_{0,\Omega_L}^2 \leq C \| e^{s \varphi(\cdot,0)} \nabla \cdot ( ( \pd_i^j  c_{\chi} ) \nabla \theta_0 ) \|_{0,\Omega_L}^2,\ i \in \N_n^*,\ j=0,1.
\ee
Further, taking $t=0$ in the first line of \eqref{s4} with $j=0$, we get that
\bel{s7}
\nabla \cdot ( c_{\chi} \nabla \theta_0 )=u_{\chi}^{(2)}(\cdot,0)+ c \nabla \chi \cdot \nabla \theta_0.
\ee
From this, \eqref{s6} and \eqref{d7} then follows that
\bea
s \| e^{s \varphi(\cdot,0)} c_{\chi} \|_{0,\Omega_L}^2 & \leq & C \left( \| e^{s \varphi(\cdot,0)} u_{\chi}^{(2)}(\cdot,0) \|_{0,\Omega_L}^2 + \| e^{s \varphi(\cdot,0)} c \nabla \chi  \cdot \nabla \theta_0 \|_{0,\Omega_L}^2 \right) \nonumber \\
& \leq & C \left( \| e^{s \varphi(\cdot,0)} u_{\chi}^{(2)}(\cdot,0) \|_{0,\Omega_L}^2 + e^{2 s \tilde{d}_{\ell}} \right). \label{s8}
\eea
Similarly, since $\nabla \cdot ( (\pd_i c_\chi) \nabla \theta_0 ) = \pd_i \nabla \cdot ( c_\chi \theta_0 ) - \nabla \cdot (c_\chi \nabla \pd_i \theta_0 )$ for every $i \in \N_n^*$, we derive from \eqref{s7} that
$$
\nabla \cdot ( (\pd_i c_\chi ) \nabla \theta_0 ) = \pd_i u_\chi^{(2)}(\cdot,0) + \pd_i ( c \nabla \theta_0 \cdot \nabla \chi ) - \nabla c_\chi \cdot \nabla \pd_i \theta_0- c_\chi \Delta \pd_i \theta_0.
$$
As a consequence we have, 
\beas
& & \| e^{s \varphi(\cdot,0)} \nabla \cdot ( (\pd_i c_\chi) \nabla \theta_0 ) \|_{0,\Omega_L}^2 \\
& \leq & C \left( \| e^{s \varphi(\cdot,0)} \pd_i u_\chi^{(2)}(\cdot,0) \|_{0,\Omega_l}^2 
+ \sum_{j=0,1} \| e^{s \varphi(\cdot,0)} \nabla^j c_\chi \|_{0,\Omega_l}^2
+ e^{2 s \tilde{d}_\ell} \right),\  i \in \N_n^*,
\eeas
according to \eqref{d7}. Summing up the above estimate over $i$ in $\N_n^*$, it follows from
\eqref{s6} that
$$
s \| e^{s \varphi(\cdot,0)} \nabla c_{\chi} \|_{0,\Omega_L}^2 \leq C \left( \| e^{s \varphi(\cdot,0)} \nabla u_\chi^{(2)}(\cdot,0) \|_{0,\Omega_L}^2 +
\sum_{j=0,1} \| e^{s \varphi(\cdot,0)} \nabla^j c_\chi \|_{0,\Omega_L}^2 
+ e^{2 s \tilde{d}_\ell} \right).
$$
This and \eqref{s8} yield the desired result.
\end{proof}

\subsection{Completion of the proof}
The proof is divided into three steps.

\noindent {\it Step 1.} The first step of the proof is to upper bound $u_\chi^{(2)}(\cdot,0)$ in the $e^{s \varphi(\cdot,0)}$-weighted $H^1(\Omega_L)$-norm topology, by the corresponding norms of $u_\chi^{(2)}$ and $u_\chi^{(3)}$ in $Q_L$, with the aid of the following technical result, borrowed from \cite{BCS12}[Lemma 3.2].
\begin{lemma}
\label{lm2}
There exists a constant $s_*>0$ depending only on $T$  such that we have
$$ \| z(\cdot,0) \|_{0,\Omega_L}^2 \leq 2 \left( s  \| z \|_{0,Q_L}^2 + s^{-1} \| \pd_t z \|_{0,Q_L}^2 \right), $$
for all $s \geq s_*$ and $z \in H^1(-T,T;L^2(\Omega_L))$.
\end{lemma}
Namely, we apply Lemma \ref{lm2} with $z = e^{s \varphi} \pd_i^j u_\chi^{(2)}$ for $i \in \N_n^*$ and $j=0,1$, getting 
$$
\| e^{s \varphi(\cdot,0)} \pd_i^j u_\chi^{(2)}(\cdot,0) \|_{0,\Omega_L}^2 \leq 2 \left( s  \| e^{s \varphi} \pd_i^j u_\chi^{(2)} \|_{0,Q_L}^2 + s^{-1} \| e^{s \varphi} \pd_i^j u_{\chi}^{(3)} \|_{0,Q_L}^2 \right),\ s \geq s_*.
$$
Summing up the above estimate over $i \in \N_n^*$ and $j=0,1$, we obtain for all $s \geq s_*$ that
\bel{s10}
\sum_{j=0,1} \| e^{s \varphi(\cdot,0)} \nabla^j u_\chi^{(2)}(\cdot,0) \|_{0,\Omega_L}^2
\leq 2 \sum_{j=0,1} \left( s \| e^{s \varphi} \nabla^j u_\chi^{(2)} \|_{0,Q_L}^2 + s^{-1} \| e^{s \varphi} \nabla^j u_\chi^{(3)} \|_{0,Q_L}^2 \right).
\ee

\noindent {\it Step 2.} The next step involves majorizing the right hand side of \eqref{s10} with
\bel{hs}
\mathfrak{h}_k(s) := \sum_{j=0,1} s^{2(1-j)} \| {\rm e}^{s \varphi} \nabla_{x,t}^j u_\chi^{(k)} \|_{0,\Sigma_L}^2,\ k=2,3.
\ee
Indeed, since $u_{\chi}^{(k)}$, for $k=2,3$,  is solution to \eqref{c18} with $c=c_1$ and $f=f_{c_\chi} - g_k$, according to \eqref{s4}, then Corollary \ref{cor-ec} yields
$$
s \sum_{j=0,1} s^{2(1-j)} \| {\rm e}^{s \varphi} \nabla_{x,t}^j u_\chi^{(k)} \|_{0,Q_L}^2  \leq C \left( \| {\rm e}^{s \varphi} f_{c_\chi}^{(k)} \|_{0,Q_L}^2 + \| {\rm e}^{s \varphi} g_k \|_{0,\tilde{Q}_{\eps}}^2+  s^3{\rm e}^{2 s \tilde{d}_\ell} \| u_\chi^{(k)} \|_{1,Q_L}^2
+ s \mathfrak{h}_k(s) \right),
$$
for $s$ large enough.
In light of \eqref{s10} this entails that
\bea 
& & \sum_{j=0,1} \| e^{s \varphi(\cdot,0)} \nabla^j u_\chi^{(2)}(\cdot,0) \|_{0,\Omega_L}^2 \nonumber \\
& \leq &
C \sum_{k=2,3} \left( \| {\rm e}^{s \varphi} f_{c_\chi}^{(k)} \|_{0,Q_L}^2 + \| {\rm e}^{s \varphi} g_k \|_{0,\tilde{Q}_{\eps}}^2 + s^3 {\rm e}^{2 s \tilde{d}_\ell} \| u_\chi^{(k)} \|_{1,Q_L}^2 + s\mathfrak{h}_k(s) \right). \label{s13}
\eea
Further, recalling \eqref{s4b}, we see from \eqref{z2} and \eqref{S100} (resp., from  \eqref{z2}, \eqref{c0b}, \eqref{S100} and \eqref{d7}) that the first (resp., second) term of the sum in the right hand side of 
\eqref{s13} is upper bounded up to some multiplicative constant, by $\sum_{j=0,1} \| {\rm e}^{s \varphi} \nabla^j c_\chi \|_{0,Q_L}^2$ (resp., ${\rm e}^{2 s \tilde{d}_\ell} ( \| u^{(k)} \|_{1,Q_L}^2 + 1)$). 
From this and Lemma \ref{lm1} then follows for $s$ sufficiently large that
\bea
& &  s \sum_{j=0,1} \| e^{s \varphi(\cdot,0)} \nabla^j c_\chi \|_{0,\Omega_L}^2  \nonumber \\
& \leq & C \left( \sum_{j=0,1} \| {\rm e}^{s \varphi} \nabla^j c_\chi \|_{0,Q_L}^2 + {\rm e}^{2 s \tilde{d}_\ell} +\sum_{k=2,3} \left( s^3 {\rm e}^{2 s \tilde{d}_\ell} \| u^{(k)} \|_{1,Q_L}^2 + s \mathfrak{h}_k(s) \right) \right). \label{3.18}
\eea
\noindent {\it Step 3.}
Finally, we notice from \eqref{c3} that
\bel{3.18b}
\| {\rm e}^{s \varphi} \nabla^j c_\chi \|_{0,Q_L} = \| \rho_s^{1 \slash 2} {\rm e}^{s \varphi(\cdot,0)} \nabla^j c_\chi \|_{0,\Omega_L},\ j=0,1,
\ee
where $\rho_s(x):=\int_{-T}^T {\rm e}^{2 s (\varphi(x,t) - \varphi(x,0))} dt = \int_{-T}^T {\rm e}^{-2 s \varphi(x,0) (1 - \exp(-\gamma t^2))} dt$ for all $x \in \Omega_L$. Bearing in mind that $\varphi(x,0) \geq \kappa:={\rm e}^{\gamma(\inf_{x'\in\omega} \left|x'-\delta a'\right|^2-L^2)} >0$ for all $x \in \Omega_L$, we get that
$$ \| \rho_s \|_{L^{\infty}(\Omega_L)} \leq \int_{-T}^T {\rm e}^{-2 s \kappa (1 - \exp(-\gamma t^2))} dt,\ s >0. $$
Therefore we have $\lim_{s \rightarrow + \infty} \rho_s = 0$, uniformly in $\Omega_L$, by the dominated convergence theorem, so we derive from
\eqref{3.18}-\eqref{3.18b} that
\bel{3.18d}
s \sum_{j=0,1} \| e^{s \varphi(\cdot,0)} \nabla^j c_\chi \|_{0,\Omega_L}^2 \leq C \left( {\rm e}^{2 s \tilde{d}_\ell} + \sum_{k=2,3} \left( s^3 {\rm e}^{2 s \tilde{d}_\ell} \| u^{(k)} \|_{1,Q_L}^2 + s \mathfrak{h}_k(s) \right) \right),
\ee
upon taking $s$ sufficiently large. With reference to \eqref{d5}--\eqref{d7}, this entails that
\bel{3.37b}
\sum_{j=0,1}  \| \nabla^j c_\chi \|_{0,\Omega_L}^2
\leq
C \sum_{k=2,3}  \left( s^2 {\rm e}^{-2s (d_\ell-\tilde{d}_\ell)} + \mathfrak{h}_k(s) \right).
\ee
Here we used \eqref{z2}-\eqref{S100} and the embedding $\Omega_{\ell} \subseteq \Omega_{L}$ in order to substitute $\Omega_{\ell}$ for $\Omega_{L}$ in the left hand side of \eqref{3.18d}.
Now, taking into account that $\tilde{d}_\ell< d_\ell$, we end up getting the desired result from \eqref{3.37b}.

\bigskip

\end{document}